\documentclass{amsart}
\usepackage{marktext} 
\usepackage{svmacro}
\usepackage{geometry}
\usepackage[utf8]{inputenc}
\usepackage{lineno}
\usepackage[backend=biber, maxbibnames=99, style=nature, sorting=nty, intitle=true]{biblatex}
\title[Local well-posedness for a critical C-F equation]{ Local mass-conserving solution for a critical Coagulation-Fragmentation equation}
\author{Hung V. Tran}
\author{Truong-Son Van}
\date{\today}



\thanks{
The work of HT is partially supported by NSF CAREER grant DMS-1843320 and a Simons Fellowship.
The work of T-SV was mostly done when he was a Hans Rademacher Instructor at the University
of Pennsylvania.
}

\address[H. V. Tran]
{
Department of Mathematics, 
University of Wisconsin Madison, Van Vleck Hall, 480 Lincoln Drive, Madison, Wisconsin 53706, USA}
\email{hung@math.wisc.edu}
\address[T.-S. Van]
{
Fulbright University Vietnam, 
105 Ton Dat Tien, District 7,
Ho Chi Minh City, Vietnam}
\email{son.van@fulbright.edu.vn}

\begin{document}

\maketitle

\begin{abstract}
    The critical coagulation-fragmentation
    equation
    with multiplicative coagulation and constant fragmentation kernels
    is known to not have  global mass-conserving solutions
     when the initial mass is greater than $1$.
     We show that for any given positive initial mass with finite second moment,
     there is a time $T^*>0$ such that the equation possesses
     a unique mass-conserving solution up to $T^*$.
     The novel idea is to singularly perturb the constant fragmentation kernel
     by small additive terms and study the limiting behavior of the solutions
     of the  perturbed system via the Bernstein transform.
\end{abstract}

\smallskip

\section{Introduction}

We study the local well-posedness of the following coagulation-fragmentation
equation (C-F)
\begin{equation}
    \label{eq:cf}
    \begin{dcases}
    \partial_t \rho (s,t) = Q_C(\rho)(s,t) +Q_F(\rho)(s,t)
    & \text{ in } (0,\infty)\times (0,\infty)\,, \\
    \rho(s,0) = \rho_0(s) & \text{ in } (0,\infty).
    \end{dcases}
\end{equation}
Here, $\rho(s,t)$ is the density of particles of size $s$ at time $t\geq 0$.
The coagulation and fragmentation terms are given by
\begin{gather*}
    Q_C(\rho)(s,t) = \frac{1}{2}\int_0^s a(s -\hat s,\hat s) \rho(s-\hat s,t) \rho( \hat s,t)\, d\hat s 
                     - \int_0^\infty a(s,\hat s) \rho(s,t) \rho(\hat s,t) \, d\hat s \,,\\
    Q_F(\rho)(s,t) = -\frac{1}{2}\int_0^s b(s -\hat s,\hat s) \rho(s,t) \, d\hat s 
                     + \int_0^\infty b(s,\hat s) \rho(s+\hat s,t)  \, d\hat s \,.
\end{gather*}
The kernels of interest are
\begin{equation*}
    a(s,\hat s) = s\hat s \,, \qquad b(s,\hat s) = 1 \quad
    \text{ for } s,\hat s > 0\,.
\end{equation*}
Let $m_k(t)$ be the $k$-th moment of  $\rho (\cdot,t)$ for $k, t\geq 0$,
that is,
\begin{equation*}
    m_k(t) = \int_0^\infty s^k  \rho(s,t) \, ds \,.
\end{equation*}
In particular the first moment $m_1(t)$ represents the total 
mass of the system~\eqref{eq:cf} at time $t\geq 0$.

It was conjectured that with this specific choice of kernels, if the system starts out
with initial mass $m_1(0)\leq 1$, then there will be a unique mass-conserving solution (see~\cite{VigilZiff1989, EscobedoLaurenccotMischlerPerthame2003}). 
It has been proven that
mass-conserving solutions cannot exist for all time when $m_1(0) > 1$ in~\cite{BanasiakLambLaurencot2019, TranVan2021}.

The history of C-F dates back more than a century ago with the work~\cite{Smoluchowski1916}.
For more extensive discussions of the subject, we refer the readers to 
the following works~\cite{Aldous1999, Costa2015, BanasiakLambLaurencot2019, TranVan2021}.
Although there has been a lot of advancement in the field, 
 a lot still remains to be discovered, including the following
question we set out to answer in this paper.

\begin{question*}
   Fix $m_1(0)>1$. Is there a time $T >0$ such that 
   the equation~\eqref{eq:cf} still has 
   a unique mass-conserving solution
   for $0\leq t < T$?
\end{question*}
The answer to such question is not obvious as there are systems of coagulation-fragmentation
equations that exhibit instantaneous mass loss, that is, $m_1(t) < m_1(0)$ for all $t>0$. 
To the best of our knowledge, past studies were divided into 
two main themes: local well-posedness of mass-conserving solutions 
for the pure coagulation equation~\cite{McLeod1962,McLeod1964,Norris2000,FournierLaurencot2006}
and local-wellposedness for weak/mild solutions for the coagulation-fragmentation equation~\cite{BanasiakLambLaurencot2019}.
Our paper is the first to study the local-wellposedness of mass-conserving solutions to a full coagulation-fragmentation equation in the situation when global mass-conserving solutions are known to not exist.

For $t\geq 0$, let $\pi(\cdot,t)$ be the distribution corresponding to the density $\rho(\cdot,t)$, i.e.,
\begin{equation*}
    \pi(s,t) = \int_{(0,s]} \rho(r,t)\, dr  \quad \text{ for } s>0\,.
\end{equation*}
By probabilistic convention, we use the same notion $\pi(\cdot,t)$ to denote the measure on $(0,\infty)$ with this distribution function.
We always use the following  notion of weak solutions to~\eqref{eq:cf}.
\begin{definition}
    \label{def:weak-sol}
    We say that $\rho$ is a weak solution of 
    the equation~\eqref{eq:cf} in the measure sense if 
    \begin{align*}
        \frac{d}{dt} \int_0^\infty \phi(s)d\pi(s,t) 
        & =  \frac{1}{2}\int_0^\infty \int_0^\infty \paren[\big]{ \phi(s+\hat s) - \phi(s) - \phi(\hat s)  }a(s,\hat s) 
        \, d\pi(s,t)\, d\pi(\hat s,t) \\
        &\quad - \frac{1}{2}\int_0^\infty \int_0^s \paren[\big]{ \phi(s) - \phi(s-\hat s) - \phi(\hat s)  } b(s-\hat s, \hat s) 
       \, d\hat s\, d\pi( s,t) \,,
    \end{align*}
    for all test functions $\phi \in \mathrm{BC}([0,\infty))\cap \mathrm{Lip}([0,\infty))$ with $\phi(0) = 0$.
    Here, $\mathrm{BC}([0,\infty))$ and $\mathrm{Lip}([0,\infty))$ 
    are the classes of bounded continuous functions and Lipchitz continuous functions on $[0,\infty)$, respectively.
\end{definition}
For $x\geq 0$, let $\phi_x(s) = 1 - e^{-xs}$ be a test function, 
and denote by $F(x,t)$ the Bernstein transform of $\rho(s,t)$,
that is,
\begin{equation*}
    F(x,t) = \mathfrak{B}[\rho(\cdot,t)](x) \defeq \int_0^\infty (1 - e^{-xs})\,  d\pi(s,t)  \,.
\end{equation*}
Define $m \defeq m_1(0)$.
If $m_1(t)  = m $ for $t\in [0,T)$ with some $T>0$, then 
$F$ satisfies the following singular Hamilton-Jacobi equation
\begin{equation}
    \label{eq:main}
    \begin{dcases}
        \partial_t F +  \frac{1}{2} (\partial_x F - m)(\partial_x F - m -1) 
        + \frac{F}{x} - m = 0 & \text{ in } (0,\infty)\times (0,T)\,, \\
        0 \leq F(x,t) \leq mx & \text{ on } [0,\infty)\times [0,T)\,, \\
        F(x,0) = F_0(x) = \mathfrak{B}[\rho_0](x) & \text{ on } [0,\infty) \,.
    \end{dcases}
\end{equation}

Our main result in this paper is the following.
\begin{theorem}
    \label{t:main}
    Let $\rho_0$ be a density function such that 
   $m=m_1(0) \geq1/2$, and $m_2(0)>0$. 
   Then, there exists a unique mass-conserving weak solution  
   to equation~\eqref{eq:cf}  for 
   $t\in [0, T^*)$, where
   \begin{equation*}
       T^* = \frac{6m}{6 m-1}\frac{1}{m_2(0)} \,.
   \end{equation*}
\end{theorem}
In the case when $a(s,\hat s) = s\hat s$ and $b(s,\hat s) = 0$, it
was shown in~\cite{MenonPego2004} that the C-F loses mass exactly after $T = m_2(0)^{-1}>0$.
Therefore, our result is consistent with that in~\cite{MenonPego2004}. 
The reason for this is
 that fragmentation helps prevent gelation (mass loss by formation of infinite-size particles). 
Equation~\eqref{eq:cf} is more complicated because of the interaction between coagulation and fragmentation kernels.
We note that  $T^*$ is not known to be sharp in Theorem~\ref{t:main}.

We briefly summarize the progress of the conjecture 
in~\cite{VigilZiff1989, EscobedoLaurenccotMischlerPerthame2003} for $m=m_1(0)\in (0,1]$.
Under certain assumptions, global existence and uniqueness of 
mass-conserving solutions when $m\leq 1/(4\log 2)$
was proven in~\cite{Laurencot2020} by the moment-bound method.
By studying equation~\eqref{eq:main}, the authors of this paper
obtained the global well-posedness for 
$m < 1/2$ in~\cite{TranVan2021}, which means that $T^*=+\infty$ in this regime.
Furthermore, while uniqueness of mass-conserving solutions for 
$m\in [1/2,1]$
was established in~\cite{TranVan2021}, 
the existence question remains an outstanding open problem.
For large time behavior results of~\eqref{eq:main}, 
see~\cite{TranVan2021,MitakeTranVan2021}.

We continue pushing limit of  Bernstein transform
to study equation~\eqref{eq:main}
and
 to establish Theorem~\ref{t:main}.
 This technique has gained fruitful results in the past for 
 so-called ``solvable kernels''~\cite{MenonPego2004,MenonPego2008, DegondLiuPego2017,TranVan2021,MitakeTranVan2021}.
 The novel idea of our approach in this work
 is to singularly perturb $b(s,\hat s) = 1$ 
 by $b^\epsilon(s,\hat s) = 1 + \epsilon(s+\hat s)$
 for $\epsilon >0$
 and study the limiting behavior of the solutions
 of the  perturbed system via the Bernstein transform.

\subsection*{Outline of the paper}
In Section~\ref{sec:approximating-frag}, we introduce an approximating system to~\eqref{eq:cf} and study various
properties of it that is inherent to~\eqref{eq:cf}.
In Section~\ref{sec:wellposed}, we give a proof of Theorem~\ref{t:main} by taking the limit
of the approximating system.
Our arguments are based on studying viscosity solutions of~\eqref{eq:main}.
Finally, in Appendix, we give a heuristic argument explaining
why one should expect that our approximating 
system possesses a mass-conserving solution.

\subsection*{Acknowledgements}
We thank Professor Philippe Laurencot for 
some discussions concerning the approximation approach.

\section{Approximating fragmentations}
\label{sec:approximating-frag}
In this section, we always assume the setting of Theorem~\ref{t:main}.
In order to study equation~\eqref{eq:cf}, we regularize it by adding a small additive term to
the fragmentation kernel. More precisely, for $\epsilon>0$, we consider

\begin{equation}
    \label{eq:cf-eps}
    \begin{dcases}
    \partial_t \rho^\epsilon (s,t) = Q_C(\rho^\epsilon)(s,t) +Q^\epsilon_F(\rho^\epsilon)(s,t)
    & \text{ in } (0,\infty)\times (0,\infty)\,, \\
    \rho^\epsilon(s,0) = \rho_0(s) & \text{ in } (0,\infty).
    \end{dcases}
\end{equation}
Here, the corresponding fragmentation kernel is
\begin{equation*}
    b^\epsilon (s, \hat s) = 1 + \epsilon(s + \hat s) \quad \text{for } s,\hat s >0\,.
\end{equation*}
Equation~\eqref{eq:cf-eps} has been shown to have a unique
mass-conserving solution by~\cite{EscobedoLaurenccotMischlerPerthame2003}.
By studying the second moment of $\rho^\epsilon$, we can show
that the second moment of the solution to equation~\eqref{eq:cf}, if exists, is finite
up to  time $T^* =  \frac{6m}{6 m-1}\frac{1}{m_2(0)}$.

Let $F^\epsilon(x,t) = \mathfrak{B}[\rho^\epsilon(\cdot,t)](x)$, the Bernstein transform of 
$\rho^\epsilon$.
Then, $F^\epsilon$ satisfies the following equation.


    \begin{equation}
    \label{eq:eps}
    \begin{dcases}
        \partial_t F^\epsilon +  \frac{1}{2} (\partial_x F^\epsilon - m)(\partial_x F^\epsilon - m -1) 
        + \frac{F^\epsilon}{x} - m 
        = \epsilon G^\epsilon(x,t)  & \text{ in } (0,\infty)^2 \,,\\
        0\leq F^\epsilon \leq mx & \text{ on } [0,\infty)^2 \,, \\
        F^\epsilon(x,0) = F_0(x) & \text{ on } [0,\infty) \,.
    \end{dcases}
    \end{equation}
    Here,
    \begin{equation*}
        G^\epsilon (x,t) \defeq 
            \frac{m_2^\epsilon(t)}{2} - \frac{\partial_x^2 F^\epsilon (x,t)}{2} 
                - \frac{1}{x}\paren[\big]{m - \partial_x F^\epsilon(x,t)} \,.
    \end{equation*}
   For derivations of equations~\eqref{eq:main} and~\eqref{eq:eps}, 
   we refer the reader to~\cite{TranVan2021,MitakeTranVan2021}.
\begin{remark}
   It is interesting to note that equation~\eqref{eq:eps} is
   a backward parabolic equation.
   The well-posedness theory for this equation from the 
   PDE viewpoint is a very interesting open question.
\end{remark}


\begin{lemma}
    \label{lem:eps-second-moment-bound}
   For $\epsilon >0$, 
    let $\rho^\epsilon$ be the mass-conserving solution to equation~\eqref{eq:cf-eps}.
   Then, for  $0 \leq t< T^*$,
    \begin{equation*}
        m_2^\epsilon(t) \leq \frac{1}{ m_2(0)^{-1} -(6m)^{-1}(6m-1) t} \leq \frac{2}{T^* - t}\,.
    \end{equation*}
\end{lemma}

\begin{proof}
    We proceed by using the method in~\cite[Theorem 3.1]{EscobedoLaurenccotMischlerPerthame2003}.
    The idea is to use a cut-off technique and the moment-bound method.
    For $n\in \N$, define
    \begin{gather*}
        a^n(s,\hat s) \defeq a(s,\hat s)\one_{[0,n]}(s+\hat s)\,, \\
        b^{\epsilon,n}(s,\hat s) \defeq b^{\epsilon}(s,\hat s)\one_{[0,n]}(s+\hat s )\,, \\
        \rho^n_0 \defeq \rho_0 \one_{[0,n]}\,.
    \end{gather*}
    Denote the coagulation and fragmentation terms corresponding the above kernels as
    $Q_C^n$ and $Q_F^{\epsilon,n}$, respectively.
    Let $\rho^{\epsilon,n}$ be the solution to the equation  
    \begin{equation}
        \label{eq:cf-eps-n}
        \begin{dcases}
        \partial_t \rho^{\epsilon,n} (s,t) = Q_C^n(\rho^{\epsilon,n})(s,t) +Q_F^{\epsilon,n}(\rho^{\epsilon,n})(s,t)
        & \text{ in } (0,\infty)\times (0,\infty)\,, \\
        \rho^{\epsilon,n}(s,0) = \rho_0^n(s) & \text{ in } (0,\infty).
        \end{dcases}
    \end{equation}

    Then, $\supp(\rho^{\epsilon,n}(\cdot,t)) \subseteq [0,n]$ for every $t\geq 0$.
    Let $\pi^{\epsilon,n}(\cdot,t)$ be the distribution corresponding to the density $\rho^{\epsilon,n}(\cdot,t)$.
    Take  $\phi(s) = s^2\one_{[0,n]}(s) + n^2 \one_{(n,\infty)}(s)$ as a test function
    in Definition~\ref{def:weak-sol}.
    Then,
    \begin{align*}
        \frac{d}{dt} m_2^{\epsilon,n}(t) 
        & =  
        \int_0^\infty \int_0^\infty s^2\hat s^2 \, \one_{[0,n]}(s+\hat s) 
        \, d\pi^{\epsilon,n}(s,t)\, d\pi^{\epsilon,n}(\hat s,t) \\
        &\quad - \int_0^\infty \int_0^{s} (s-\hat s)\hat s \one_{[0,n]}(s)
       \, d\hat s\,  (1+\epsilon s)\, d\pi^{\epsilon,n}( s,t) \\
        & \leq 
        \int_0^\infty \int_0^\infty s^2\hat s^2  
        \, d\pi^{\epsilon,n}(s,t)\, d\pi^{\epsilon,n}(\hat s,t)  - \int_0^\infty \frac{s^3}{6} \, \one_{[0,n]}(s) 
        \, d\pi^{\epsilon,n}(s,t)\\
        &= m_2^{\epsilon,n}(t)^2 -\frac{m_3^{\epsilon,n}(t)}{6} \leq \left( 1- \frac{1}{6m}\right) m_2^{\epsilon,n}(t)^2\,,
    \end{align*}
    where we used the Cauchy-Schwarz inequality in the last line
    \[
    m m_3^{\epsilon,n}(t) \geq m_1^{\epsilon,n}(t) m_3^{\epsilon,n}(t) \geq m_2^{\epsilon,n}(t)^2 .
    \]
    Therefore, for $t <  \frac{6m}{6 m-1}\frac{1}{m_2^n(0)}$,
    \begin{equation*}
        m_2^{\epsilon,n}(t) \leq \frac{1}{ m_2^{n}(0)^{-1} - (6m)^{-1}(6m-1)t}\,.
    \end{equation*}
    Letting $n\to\infty$ and using the compactness result in~\cite[Remark 3.10]{EscobedoLaurenccotMischlerPerthame2003}, 
    we have
    \begin{equation*}
        m_2^{\epsilon}(t) \leq \frac{1}{ m_2(0)^{-1} - (6m)^{-1}(6m-1) t}\,,
    \end{equation*}
    for $t <  \frac{6m}{6 m-1}\frac{1}{m_2(0)}$.
      The proof of the first inequality is finished by picking $T^* = \frac{6m}{6 m-1}\frac{1}{m_2(0)}$.

      For the second inequality, we recall that by the assumptions in Theorem~\ref{t:main},
      $m\geq 1/2$.
      Therefore,
      \begin{equation*}
          \frac{1}{ m_2(0)^{-1} - (6m)^{-1}(6m-1) t} = \frac{1}{ (6m)^{-1}(6m-1) (T^* - t)} 
          \leq \frac{2}{T^* - t} \,,
      \end{equation*}
      as desired.
\end{proof}

\begin{remark}
    \label{rem:general}
    While Lemma~\ref{lem:eps-second-moment-bound} provides the crucial 
    estimate for us to pass the equation to the limit below, it is not
    enough to conclude that the limiting equation would have local 
    mass-conserving solutions.
    Indeed, it is evidently clear that we would be able to control 
    the second moment of the limiting C-F equation for $0<t<T^*$.
    However, for a more general fragmentation kernel,
    it is unclear whether mass is lost 
    via fragmentation (dust) or not.
    The subtlety here lies in the fact that in order to achieve
    mass-conserving solutions, one needs to be able to control
    the strength of both coagulation and fragmentation kernels,
    which are of equal importance.
    Extreme care needs to be paid when there is a competition between coagulation and fragmentation.
    The use of Bernstein transforms and the study of the regularity
    of the Hamilton-Jacobi equation help us bypass this fine point.
\end{remark}

We now let $\epsilon \to 0$ in equation~\eqref{eq:eps} to obtain a viscosity solution
of~\eqref{eq:main}.

\begin{lemma}
    \label{lem:convergence-F}
    For each $\epsilon >0$, let $F^\epsilon$ be the smooth solution to equation~\eqref{eq:eps}.
    Then, there exists $\bar F\in C([0,\infty)\times [0,T^*))$ such that, locally uniformly for  $(x,t) \in [0,\infty)\times[0,T^*)$,
    \begin{equation*}
        \lim_{i \to \infty} F^{\epsilon_i}(x,t) = \bar F(x,t)\,,
    \end{equation*}
    for some sequence $\set{\epsilon_i} \to 0$.
    Furthermore, $\bar F$ is a viscosity solution to~\eqref{eq:main}.
\end{lemma}

\begin{proof}
    We recall equation~\eqref{eq:eps} 
    \begin{equation*}
    \begin{dcases}
        \partial_t F^\epsilon +  \frac{1}{2} (\partial_x F^\epsilon - m)(\partial_x F^\epsilon - m -1) 
        + \frac{F^\epsilon}{x} - m 
        = \epsilon G^\epsilon(x,t)  & \text{ in } (0,\infty)^2 \,,\\
        0\leq F^\epsilon \leq mx & \text{ on } [0,\infty)^2 \,, \\
        F^\epsilon(x,0) = F_0(x) & \text{ on } [0,\infty) \,,
    \end{dcases}
    \end{equation*}
    where
    \begin{equation*}
        G^\epsilon (x,t) = 
            \frac{m_2^\epsilon(t)}{2} - \frac{\partial_x^2 F^\epsilon (x,t)}{2} 
                - \frac{1}{x}\paren[\big]{m - \partial_x F^\epsilon(x,t)} \,.
    \end{equation*}
    Recall that  $F^\epsilon$ is the Bernstein transform of $\rho^\epsilon$, i.e.,
    \begin{equation*}
        F^\epsilon(x,t) = \mathfrak{B}[\rho^\epsilon(\cdot,t)](x)\,.
    \end{equation*}
    Therefore, for $(x,t) \in [0,\infty)^2$, we have
    \begin{gather*}
         0\leq \partial_x F^\epsilon(x,t)
         = \int_0^\infty s e^{-sx} d\pi^{\epsilon}(s,t)
        \leq m \,, \\
          \abs{\partial_x^2 F^\epsilon(x,t)}
         = \int_0^\infty s^2 e^{-sx} d\pi^{\epsilon}(s,t)
        \leq m_2^\epsilon(t) \,.\\
    \end{gather*}
    Here, $\pi^{\epsilon}(\cdot,t)$ is the distribution corresponding to the density $\rho^{\epsilon}(\cdot,t)$.
    In particular,
    \begin{equation*}
        \partial_x F^\epsilon(0,t) = m\,.
    \end{equation*}
    By the mean value theorem,
    we have that for every $\epsilon>0$,
     $T< T^*$ and $(x,t)\in [0,\infty)\times [0,T]$,
    there exists $\theta \in (0,1)$ such that
    \begin{equation*}
        G^\epsilon(x,t) =  
            \frac{m_2^\epsilon(t)}{2} - \frac{\partial_x^2 F^\epsilon (x,t)}{2} 
            + \partial_x^2 F^\epsilon(\theta x,t)\,.
    \end{equation*}
    Therefore, by Lemma~\ref{lem:eps-second-moment-bound},
    \begin{equation}
        \label{ine:est-G}
        \sup_{(x,t)\in [0,\infty)\times [0,T]} \abs{G^\epsilon(x,t)} \leq \frac{5}{T^* - T}\,.
    \end{equation}
    Furthermore, we also have
    \begin{equation*}
        \abs[\Big]{\frac{1}{2} (\partial_x F^\epsilon - m)(\partial_x F^\epsilon - m -1) 
        + \frac{F^\epsilon}{x} - m} \leq  \frac{m(m+3)}{2}\,.
    \end{equation*}
    As a consequence, 
    for $(x,t)\in [0,\infty)\times [0,T]$ and $\epsilon \in (0,1)$,
    \begin{equation*}
        \abs{\partial_t F^\epsilon(x,t)} + \abs{\partial_x F^\epsilon(x,t)} \leq
        \frac{m(m+5)}{2} + \frac{5}{T^* - T}\,.
    \end{equation*}
    By the Arzel\`a-Ascoli theorem, there exist a function $\bar F\in C([0,\infty)\times [0,T^*))$  
    and a sequence $\set{\epsilon_i}\to 0$ such that
     $F^{\epsilon_i} \to \bar F$ locally uniformly on $[0,\infty)\times[0,T^*)$.

    For $\epsilon \geq 0$, write
    \begin{equation*}
        H^\epsilon(x,t,p,u) \defeq \frac{(p-m)(p-m-1)}{2} + \frac{u}{x} - m - \epsilon G^\epsilon(x,t)\,.
    \end{equation*}
    Equations~\eqref{eq:eps} can be rewritten as 
    \begin{equation*}
        \partial_t F^\epsilon + H^\epsilon(x,t, \partial_x F^\epsilon, F^\epsilon) = 
        0 \quad \text{ in } (0,\infty)\times (0,T^*)\,.
    \end{equation*}
    Furthermore, we have that by estimate~\eqref{ine:est-G},
    \begin{equation*}
        H^\epsilon \to H^0
    \end{equation*}
    locally uniformly in $(0,\infty)\times (0,T^*)\times [0,m]\times (0,\infty)$.
    Thus, by the stability of viscosity solutions (see, e.g.,~\cite{Tran2021}), 
    $\bar F$ is a viscosity solution of equation~\eqref{eq:main}.
\end{proof}

\begin{lemma}
    \label{lem:convergence-DF}
    Let 
     $\bar F$ be a viscosity solution to equation~\eqref{eq:main} given
    by Lemma~\ref{lem:convergence-F}.
    Then, in the viscosity sense, in $(0,\infty)\times (0,T)$ for $T< T^*$, we have
    \begin{equation}
        \label{eq:firstDerivativeBound2}
        0 \leq \partial_x \bar F \leq m\,,
    \end{equation}
    \begin{equation}
        \label{eq:timeBound}
       \abs{\partial_t \bar F} \leq  \frac{m(m+5)}{2} + \frac{5}{T^* - T}\,,
    \end{equation}
    and 
    \begin{equation}
        \label{eq:secondDerivativeBound3}
        -\frac{2}{T^* - T} \leq \partial_x^2 \bar F \leq 0 \,.
    \end{equation}
\end{lemma}

\begin{proof}
    This is a  consequence of 
    $0 \leq \partial_x F^\epsilon\leq m$, 
    Lemmas~\ref{lem:eps-second-moment-bound} and~\ref{lem:convergence-F}.
    In particular, the inequalities~\eqref{eq:firstDerivativeBound2}
    and~\eqref{eq:timeBound} are 
    straightforward.

    Let us now prove~\eqref{eq:secondDerivativeBound3}.
    As $F^\epsilon$ is the Bernstein transform of $\rho^\epsilon$, we have that
    for $(x,t)\in (0,\infty)\times (0,T)$,
    \begin{equation*}
        -\frac{2}{ T^*- T} \leq 
        - m_2^\epsilon(t) = \partial_x^2 F^\epsilon(0,t) \leq \partial_x^2 F^\epsilon(x,t) \leq 0\,.
    \end{equation*}
    
    Furthermore, for $h>0$, there exist $\theta, \tilde \theta \in (0,1)$ such that
    \begin{align*}
        0&\geq
        \frac{ F^\epsilon(x+ 2h,t) + F^\epsilon (x,t) - 2F^\epsilon(x+h,t) }{h^2} \\
         & = \frac{ \partial_x^2 F^\epsilon(x+ h +\theta h,t)h^2 + \partial_x^2 F^\epsilon(x + \tilde \theta h,t)h^2}{2h^2}\\
         & = \frac{1}{2}(\partial_x^2 F^\epsilon(x+ h +\theta h,t) + \partial_x^2 F^\epsilon(x + \tilde \theta h,t))\\
         & \geq 
        \partial_x^2 F^\epsilon (x,t) 
        \geq
        -\frac{2}{ T^* - T} \,.
    \end{align*}
    In the above, we used the fact that $\partial_x^3 F^\epsilon \geq 0$.
    Letting $\epsilon \to 0$, we obtain
    \begin{equation*}
        0\geq \frac{ \bar F(x+ 2h,t) + \bar F(x,t) - 2\bar F(x+h,t) }{h^2}  \geq -\frac{2}{ T^* - T} \,.
    \end{equation*}
    Inequality~\eqref{eq:secondDerivativeBound3} follows immediately by letting $h\to 0$.
\end{proof}

In order to show the uniqueness of solutions to equation~\eqref{eq:main}, we need the following
comparison principle.
A similar result was proven in~\cite{TranVan2021}. We provide the details here
for self-containment.
  \begin{lemma}[Comparison Principle for~\eqref{eq:main}] \label{lem:CP}
      For $T\in (0,T^*)$,
    let $u$ be a sublinear viscosity subsolution and $v$ be a sublinear viscosity supersolution to equation~\eqref{eq:main}, respectively. Then $u\leq v$.
  \end{lemma}

  \begin{proof}
    Since \eqref{eq:main} is singular at $x=0$, we cut off its singularity  by introducing a sequence of function $\set{\varphi_n}$, where
    \begin{equation*}
    \varphi_n(x) = \max\left\{ \frac{1}{n},x \right\} \quad \text{ for all } x \in [0,\infty)\,.
    \end{equation*}
    For each $n\in \N$,
    we consider the following approximating Hamilton-Jacobi equation
    \begin{equation} \label{e:cutoff}
      \begin{cases}
        \partial_t F + \frac{1}{2} (\partial_x F - m) ( \partial_x F - m - 1) + \frac{F}{\varphi_n(x)}  -m = 0 \quad &\text{ in } (0,\infty)\times (0,T)\,,\\
        F(x,0) = F_0(x) \quad &\text{ on } [0,\infty)\,,\\
        F(0,t) = 0 \quad &\text{ on } [0,\infty)\,,
      \end{cases}
    \end{equation}
    We claim that  $u$ is a subsolution, and $v^n \defeq v + \frac{m}{n}$ is a supersolution to equation~\eqref{e:cutoff}, respectively.
    It is clear to see that $u$ is a subsolution. 
    To check that $v$ is a supersolution, we note that,
    \begin{equation*}
      \frac{v + \frac{m}{n}}{\varphi_n} -m = \begin{cases}
        \frac{v}{x} + \frac{m}{nx} - m\geq \frac{v}{x} - m, &\text{ for } x \geq \frac{1}{n} \,, \\
        nv + m -m \geq  \frac{v}{x} -m, &\text{ for } x < \frac{1}{n} \,.
      \end{cases}
    \end{equation*}
    Therefore, 
    \begin{align*}
     & \partial_t v^n + \frac{1}{2} (\partial_x v^n - m) ( \partial_x v^n -m - 1) + \frac{v^n}{\varphi_n(x)} - m \\
     &\geq  \partial_t v + \frac{1}{2} (\partial_x v - m) ( \partial_x v -m - 1) + \frac{v}{x} - m  \geq 0
    \end{align*}
    in the viscosity sense. By the classical theory of viscosity solution applied to equation~\eqref{e:cutoff}, we deduce that 
    \begin{equation*}
      u \leq v^n \,.
    \end{equation*}
    As $v^n \to v$  uniformly when $n\to \infty$, we then conclude
    \begin{equation*}
      u \leq v \quad \text{ on } [0,\infty)\times [0,T) \,,
    \end{equation*}
   as desired. 
  \end{proof}
\begin{corollary}
    \label{cor:unique}
    Let $\bar F$ be the function as in Lemma~\ref{lem:convergence-F}.
    Then $\bar F$ is the unique sublinear  viscosity solution
    to equation~\eqref{eq:main} on $[0,\infty)\times [0,T)$
    for $T < T^*$.
    As a consequence, the convergence in Lemma~\ref{lem:convergence-F} is in full sequence, that is,
    locally uniformly on $[0,\infty)\times [0,T^*)$, we have
    \begin{equation}
        \label{eq:full-convergence}
       \lim_{\epsilon \to 0} F^\epsilon = \bar F\,.
    \end{equation}
\end{corollary}
\begin{proof}
    First, note that $\bar F \geq 0$ as it is a subsequential limit of $F^{\epsilon_i}\geq 0$.
    By~\eqref{eq:timeBound},
    \begin{equation*}
        \bar F (x,t) \leq\paren[\Big]{ \frac{m(m+5)}{2} + \frac{5}{T^* - T}} t + F_0(x)\,,
    \end{equation*}
    for  $(x,t)\in [0,\infty)\times [0,T)$.
    Therefore, $\bar F$ is sublinear on $[0,\infty)\times [0,T)$.
    By Lemma~\ref{lem:CP}, $\bar F$ is the unique sublinear viscosity solution to equation~\eqref{eq:main}.
    By this uniqueness, \eqref{eq:full-convergence} follows immediately.
\end{proof}

\section{Local existence of solutions to the C-F equation~\eqref{eq:cf} } 
\label{sec:wellposed}

We now  prove our main result, Theorem~\ref{t:main}.
Throughout this section, we will always assume the setting of Theorem~\ref{t:main}.
Let $F=\bar F$ be the sublinear viscosity solution to equation~\eqref{eq:main} found in
Section~\ref{sec:approximating-frag}.
By Lemma~\ref{lem:convergence-DF}, 
we already have that $F \in C^{1,1}((0,\infty)\times (0,T^*))$. 
Let us now use this result to yield further that $F \in C^\infty((0,\infty)\times (0,T^*))$.

\begin{proposition} \label{prop:F-smooth}
    Let $F=\bar F$ be the sublinear viscosity solution to equation~\eqref{eq:main} for $T=T^*$. 
    Then, $F \in C^\infty((0,\infty)\times ( 0, T^*))$.
\end{proposition}

\begin{proof}
  We proceed by using the method of characteristics (see~\cite[Chapter 3]{Evans2010}).
  Fix $t\in (0,T^*)$ and
  denote by $X(x,s)$ the characteristic at time $s\in [0,t]$
  starting from $x> 0$, that is, $X(x,0)=x$.
  Set $P(x,s)=\partial_x F(X(x,s),s))$, and $Z(x,s)=F(X(x,s),s)$.
  When there is no confusion, 
  we just write $X(s),P(s),Z(s)$ instead of 
  $X(x,s), P(x,s), Z(x,s)$, respectively.
  Then, $X(0)=x$, $P(0)=\partial_x F_0(x)$, $Z(0)=F_0(x)$.
  We have the following Hamiltonian system
  \[
  \begin{cases}
  \dot X = \partial_p H(P(s),Z(s),X(s)) = P(s) - \left(m+\frac{1}{2}\right)\,,\\
  \dot P = -\partial_x H - (\partial_z H)P = \frac{Z(s)}{X(s)^2}- \frac{P(s)}{X(s)}\,,\\
  \dot Z = P\cdot \partial_pH - H = \frac{P(s)^2}{2} - \frac{Z(s)}{X(s)} + \frac{m(1-m)}{2}\,.
  \end{cases}
  \]
  Note first that 
  $F \in C^{1,1}((0,\infty)\times (0,T^*)) $, 
  and also $0 \leq \partial_x F \leq m$ thanks to Lemma~\ref{lem:convergence-DF}.
    Therefore,
  \begin{equation} \label{eq:X dot}
        - \left(m+\frac{1}{2}\right) \leq \dot X \leq -\frac{1}{2}\,.
  \end{equation}
  Besides, the concavity of $F$ in $x$ yields further that
  \[
   \dot P =  \frac{Z(s)}{X(s)^2}- \frac{P(s)}{X(s)} 
   =\frac{1}{X(s)} \left( \frac{F(X(s),s)}{X(s)} 
   - \partial_x F(X(s),s) \right) \geq 0\,.
  \]
  
  Let us now show that $\{X(x,\cdot)\}_{x \in (0,\infty)}$ are well-ordered in $(0,\infty)\times (0,T^*)$, and none of these two characteristics intersect. 
  Assume otherwise that $X(x,s)=X(y,s) > 0$ 
  for some $x \neq y$ and $s\in (0,t]$.
  As $F \in C^{1,1}((0,\infty)\times (0,T^*))$, $\partial_x F(X(x,s),s)$ is uniquely defined, and therefore,
  \[
  P(x,s)=P(y,s)= \partial_x F(X(x,s),s) \quad \text{ and } \quad Z(x,s)=Z(y,s) = F(X(x,s),s)\,.
  \]
  
    
  Hence, $(X,P,Z)(x,s)=(X,P,Z)(y,s)$, and this contradicts the uniqueness of 
  solutions to the Hamiltonian system on $[0,s]$ as we reverse the time.

    By Lemma~\ref{lem:convergence-DF}, we have that for 
    $t< T^*$ and $(x,s)\in (0,\infty)\times [0,t]$, 
  \[
        -\frac{2}{ T^* - t} 
   \leq \partial^2_x F(x,s) \leq 0  
  \]
  in the viscosity sense.
  We differentiate the first equation in the 
  Hamiltonian system with respect to $x$ and 
  use the fact that $P(x,s)=\partial_x F(X(x,s),s)$ 
  to yield that
  \[
  \partial_x \dot X(x,s) = \partial_x P(x,s) = \partial^2_x F(X(x,s),s)\cdot  \partial_x X(x,s) \geq
        -\frac{2}{ T^* - t} \partial_x X(x,s).
  \]
  Thus, $\partial_x X(x,s)$ satisfies a differential inequality, and in particular, 
  \[
  s\mapsto e^{\frac{2s}{ T^* - t}} \partial_x X(x,s) 
  \qquad \text{ is nondecreasing on } [0,t].
  \]
  Therefore, $\partial_x X(x,s) > 0$ for all $(x,s)\in (0,\infty)\times [0,t]$ as
  $\partial_x X(x,0) = 1$.
  By the inverse function theorem, $X^{-1}(\cdot,s)$ is then locally smooth, and
  \[
  F(x,s) = Z(X^{-1}(x,s),s)
  \]
 is smooth as $Z$ is also smooth.
The proof is complete.
\end{proof}


To show the absolute-monotone property of $F$, we exploit the approximating functions $\set{F^\epsilon}_{\epsilon >0}$,
which are Bernstein functions themselves.
 This allows us to  completely avoid the technical tour de force
as in~\cite[Proposition 3.10]{TranVan2021}.

\begin{lemma}
    \label{lem:bernstein-property}
    Let $F=\bar F$ be the sublinear viscosity solution to equation~\eqref{eq:main} for $T=T^*$. 
    Then,
    for every $(x,t)\in (0,\infty)\times (0,T^*)$ and $k\in \N$,
\begin{equation}
    \label{eq:bernstein}
    (-1)^{k-1} \partial_x^k   F (x,t) \geq 0 \,.
\end{equation}
\end{lemma}
\begin{proof}
    Fix $k\in \N$ and $t\in (0,T^*)$.
    For each $\epsilon>0$, let $F^\epsilon$ be the  solution to equation~\eqref{eq:eps}.
   For every test function $\varphi \in C^\infty_c((0,\infty))$ with $\varphi \geq 0$,
   we have that
   \begin{equation*}
      0\leq \int_0^\infty (-1)^{k-1} \partial_x^k F^\epsilon(x,t) \varphi(x) \, dx
       = \int_0^\infty (-1)  F^\epsilon(x,t) \partial_x^k \varphi(x) \, dx\,.
   \end{equation*}
   Letting $\epsilon\to 0$, we have
   \begin{equation*}
        \int_0^\infty (-1)  F(x,t) \partial_x^k \varphi(x) \, dx\geq 0\,.
   \end{equation*} 
   As $F\in C^\infty((0,\infty)\times (0,T^*))$, we integrate by parts once again to get
    \begin{equation*}
       \int_0^\infty (-1)^{k-1} \partial_x^k F(x,t) \varphi(x) \, dx
       \geq 0\,,
    \end{equation*}
    from which~\eqref{eq:bernstein} follows.
\end{proof}

\begin{proof}[Proof of Theorem~\ref{t:main}]
    Let $F=\bar F$ be the sublinear viscosity solution to equation~\eqref{eq:main} for $T=T^*$. 
    By Proposition~\ref{prop:F-smooth} and Lemma~\ref{lem:bernstein-property}, we deduce that
    for each $t\in (0,T^*)$, $F(\cdot,t)$ is a Bernstein function.
    Theorem~\ref{t:main} follows from the fact that for each Bernstein function $f$, 
    there exists a Borel measure $\mu$ so that $\mathfrak{B}[\mu] = f$ (see~\cite[Appendix]{TranVan2021}).
\end{proof}

\appendix
\section{Some moment bounds for solutions to~\eqref{eq:cf-eps}}
In this appendix, we will demonstrate a heuristic understanding on why equation~\eqref{eq:cf-eps}
is well-posed for all $t>0$, which was proven in~\cite{EscobedoLaurenccotMischlerPerthame2003}.
This essentially comes from the ability to control all the moments
based on certain differential inequalities.
Here, the strong fragmentation term plays a crucial role.
The argument here follows that in~\cite{EscobedoLaurenccotMischlerPerthame2003}.

Assume that $m_k(0) < \infty$ for all $k\in \N$.
Heuristically, we consider test functions $\phi(s) = s^k$
in Definition~\ref{def:weak-sol}
to read off the information about $m^\epsilon_k(t)$.
We will demonstrate how this is done for $k=2,3$.
Higher moments could be bounded in a similar manner inductively.

For $k=2$, by using $\phi(s)= s^2$, we get that
\begin{equation*}
    \frac{d}{dt} m^\epsilon_2(t)
    = 
    m_2^\epsilon(t)^2 - \frac{1}{6} m_3^\epsilon(t) - \frac{\epsilon}{6} m_4^\epsilon(t)\,.
\end{equation*}
We wish to control $m_2^\epsilon(t)^2$ to prevent blow-up in finite time of $m_2^\epsilon(t)$.
By H\"older's inequality, we have
\begin{align*}
   \paren[\Big]{\int_0^\infty \paren[\Big]{ s^{4/3} \rho_\epsilon^{1/3}}^3\, ds}^{1/3}
   \paren[\Big]{\int_0^\infty \paren[\Big]{ s^{2/3} \rho_\epsilon^{2/3}}^{3/2}\, ds}^{2/3}
   \geq
   \int_0^\infty s^2\rho_\epsilon \, ds\,.
\end{align*}    
Therefore,
\begin{equation*}
    m^\epsilon_4(t)  \geq \frac{m^\epsilon_2(t)^3}{m^2} \,.
\end{equation*}
This implies that
\begin{equation*}
    \frac{d}{dt} m^\epsilon_2(t)
    \leq 
    m_2^\epsilon(t)^2  - \frac{\epsilon}{6} \frac{m^\epsilon_2(t)^3}{m^2} \leq C_{\epsilon,2}\,,
\end{equation*}
for some $C_{\epsilon,2} >0$.
Therefore, for $t\geq 0$,
\begin{equation*}
    m_2^\epsilon(t) \leq m_2(0) +  C_{\epsilon,2} t\,.
\end{equation*}

For $k= 3$, using $\phi(s) = s^3$, we have 
\begin{equation*}
    \frac{d}{dt} m_3^\epsilon(t) 
    =
    3 m_3^\epsilon(t) m_2^\epsilon(t) - \frac{1}{12} m_4^\epsilon(t)
    -\frac{\epsilon}{12} m_5^\epsilon(t) \,.
\end{equation*}

By the Cauchy-Schwarz inequality, we have
\begin{equation*}
    m_5^\epsilon(t) m_1^\epsilon(t) \geq m_3^\epsilon(t)^2\,.
\end{equation*}
Therefore, 
\begin{equation*}
    m_5^\epsilon(t)  \geq \frac{m_3^\epsilon(t)^2}{m}\,,
\end{equation*}
and
\begin{equation*}
    \frac{d}{dt} m_3^\epsilon(t) 
    \leq 3(m_2(0) + C_{\epsilon,2}t) m_3^\epsilon(t) - \frac{\epsilon}{12 m} m_3^\epsilon(t)^2 
    \leq C_{\epsilon,3}(t+1)^2\,,
\end{equation*}
for some $C_{\epsilon,3}>0$.
Thus, for $t\geq 0$,
\begin{equation*}
    m_3^\epsilon(t) \leq m_3(0) + C_{\epsilon,3}(t+1)^3\,.
\end{equation*}


%
%
%
\printbibliography 

\end{document}